\documentclass[11pt]{amsart}
\usepackage[numbers]{natbib}
\usepackage{amsmath}
\usepackage{amssymb}
\usepackage{graphicx}
\usepackage[numbers]{natbib}%
\usepackage{epstopdf}
\usepackage[figuresright]{rotating}
\usepackage{graphics}
\usepackage{amssymb}
\usepackage{eucal}
\usepackage{amsfonts}
\usepackage{url}
\usepackage{algorithmic}
\usepackage{algorithm}
\usepackage{amsthm}

\newtheorem{theorem}{Theorem}[section]

\newtheorem{lemma}[theorem]{Lemma}

\theoremstyle{definition}
\newtheorem{definition}[theorem]{Definition}
\newtheorem{example}[theorem]{Example}
\newtheorem{remark}[theorem]{Remark}

\numberwithin{equation}{section}

\title[  A string averaging method based on strictly quasi-nonexpansive ...]{ A string averaging method based on strictly quasi-nonexpansive operators with generalized relaxation}

\author[T. Nikazad]{Touraj Nikazad}

\address[T. Nikazad]{School of Mathematics, Iran University of Science and Technology, 16846-13114 Tehran, Iran}
\email{{\tt tnikazad@iust.ac.ir}}

\author[M. Mirzapour]{Mahdi Mirzapour$^{\ast}$}
\address[M. Mirzapour]{Department of Mathematics, Faculty of Sciences, Bu-Ali Sina University, 65178-38695 Hamedan, Iran}
\email{\tt m.mirzapour@basu.ac.ir}

\keywords{convex feasibility problem; cutter operator; strictly
quasi-nonexpansive; string averaging; generalized relaxation;
extrapolation}

\subjclass[2010]{90C25; 47J25; 49M20}

\thanks{$^{\ast}$Corresponding author}

\begin{document}

\begin{abstract}
We study a fixed point iterative method based on generalized
relaxation of strictly quasi-nonexpansive operators. The iterative
method is assembled by averaging of strings, and each string is
composed  of finitely many strictly quasi-nonexpansive operators.
To evaluate the study, we examine a wide class of iterative
methods for solving linear systems of equations (inequalities) and
the subgradient projection method for solving nonlinear convex
feasibility problems. The mathematical analysis is complemented by
some experiments in image reconstruction from projections and
classical examples, which illustrate the performance using
generalized relaxation.
\end{abstract}

\maketitle


\section{Introduction}
Finding a common point in finitely many closed convex sets,
which is called the convex feasibility problem,  arises in many
areas of mathematics and physical sciences. To solve such
problems, using iterative projection methods has been suggested by
many researchers, see, e.g., \cite{BB}. Since the computing of
projections is expensive, it is advised to use, easy computing
operators such as subgradient projections and $\mathcal{T}-$class
operator, which was introduced and investigated in \cite{BC} and
studied in several research works as \cite{BCk,CC} and references
therein. This class, which is named {\it cutter} by some authors,
see, e.g., \cite{CC}, contains projection operators, subgradient
projections, firmly nonexpansive operators, the resolvents of
 maximal monotone operators and strongly quasi-nonexpansive
operators but it is a subset of strictly quasi-nonexpansive (sQNE)
operators, see \cite[p. 810]{CC} and \cite{BCk,C2013} for more
details.

We concentrate on the set of sQNE operators, which involves cutter
and paracontracting operators. It should be noted that a wide
range of iterative methods, for solving linear systems of
equations, is based on sQNE operators, see \cite[lemma 4]{NDH2012}
and
\cite{CGG2001,CE2002,CEHN2008,DHC2009,aleyner2008block,EHL1981,E1980,EN2009,JW2003,bargetz2018linear,CRZ2020,CGRZ2020}.
Furthermore, applications of such iterations appear in signal
processing, system theory, computed tomography, proton
computerized tomography and other areas. Therefore, our analysis
can be applied to all the research works we mentioned.

Breaking huge-size problems into smaller ones is a natural way to
partially reduce computational time. Connecting small problems can
play an important role from a computational point of view. The
algorithmic structures are full or block sequential
(simultaneous) methods and more general constructions, see
\cite{string2001}, which is based on string averaging. These
kinds of algorithms are particularly suitable for parallel
computing and therefore have the ability to handle huge-size
problems such as deblurring problems, microscopy, medical and
astronomic imaging, geophysical applications, digital
tomosynthesis, and other areas. However, the string averaging
process was used and analyzed in many research works as
\cite{string2001,dynsumper2013,bauschke2004projection,inco-string,nikazad2016convergence,reich2016modular,BRZ2018}. They used a
relaxed version of the projection operators in the string
averaging procedure whereas we extend it to generalized relaxation
of sQNE operators. The generalized relaxation strategy is studied
for nonexpansive operators in \cite{C2010}, for cutter operators
in \cite{CC} and recently for strictly relaxed cutter operators in
\cite{NM2017Japan}. We have to mention that any strictly relaxed cutter
operator is strictly quasi-nonexpansive operator \cite[Remark
2.1.44.]{C2013} but the converse is not true in general, see
Example \ref{exam1}.

We analyze a fixed point iteration method based on generalized
relaxation of an sQNE operator which is constructed by averaging of
strings and each string is a composition of finitely many sQNE
operators. Our analysis indicates that the generalized relaxation
of cutter operators is inherently able to provide more acceleration
comparing with \cite{CC}.

The paper is organized as follows. In Section \ref{pre} we recall
some definitions and properties of sQNE operators. We
define a string averaging process and its convergence analysis in
Section \ref{main}. The applicability of the main result is examined
in Section \ref{apl} by employing state-of-the-art iterative
methods. Section \ref{num} presents some numerical experiments in
the field of image reconstruction from projections and classical
examples.

\section{Preliminaries and Notations}\label{pre}
Throughout this paper, we assume $T:\mathcal{H} \rightarrow
\mathcal{H}$ with nonempty a fixed point set, i.e., $Fix T\not =
\emptyset,$ where $\mathcal{H}$ is a Hilbert space. Also $Id$
denotes the identity operator on $\mathcal{H}$. First we recall some
definitions from \cite{C2013} which will be useful for our future
analysis.
\begin{definition}\label{tarifstrictly}
An operator $T$ is {\it quasi-nonexpansive} if
\begin{equation}\label{t1}
\|Tx-z\|\leq\|x-z\|
\end{equation}
for all $x\in \mathcal{H}$ and $z\in Fix T.$ Also, one may use the
term sQNE by replacing strict inequality in (\ref{tarifstrictly}),
i.e., $\|Tx-z\|<\|x-z\|$ for all $x\in \mathcal{H}\backslash Fix
T$ and $z\in Fix T.$ Moreover, a continuous sQNE operator is
called paracontracting, see \cite{CR1996,EKN}.
\end{definition}
\begin{remark}\label{rem1}
A simple calculation shows that the inequality
(\ref{tarifstrictly}) is equivalent with
$$\left<z-\frac{T(x)+x}{2} , T(x)-x\right>\geq 0.$$
\end{remark}
Another useful class of operators is the class of cutter operator,
namely
\begin{definition}
An operator $T:\mathcal{H} \rightarrow \mathcal{H}$ with nonempty
fixed point set is called  $cutter$ if
\begin{equation}\label{cutterdefinition}
\left<x-T(x), z-T(x)\right> \leq 0
\end{equation}
for all $x \in \mathcal{H}$ and all $z\in FixT$.
\end{definition}
However, the set of cutter operators is not necessarily closed
with respect to the composition of operators,  whereas the class of sQNE
operators is, see \cite[Theorem 2.1.26]{C2013}. Furthermore, based on \cite[Theorem 2.1.39 and Remark 2.1.44]{C2013}, any
cutter operator belongs to the set of sQNE operators.

We next give an example which is neither paracontracting operator
nor cutter operator but that is an sQNE operator.
\begin{example}\label{exam1}
Let $T:\mathbb{R} \rightarrow \mathbb{R}$  such that
\begin{displaymath}
T(x) = \left\{ \begin{array}{ll}
-\frac{1}{2}x, & x\in Q \vspace{0.1cm}\\
-\frac{1}{3}x, & x\in Q^c
\end{array} \right.
\end{displaymath}
where $Q$ and $Q^{c}$ denote rational and irrational numbers,
respectively. The discontinuity of the operator gives that the
operator is not paracontraction. We check the property of cutter
operators, namely,
\begin{equation}
\left< x-T(x), z-T(x)\right> \leq 0 ~for~ all ~z~\in~FixT=\{0\}
~\textrm{ and } x \in \mathbb{R}.
\end{equation}
Choosing $x=1$ and $z=0$ leads to $\left< 1-T(1), 0-T(1)\right>
=1/4
>0$. Therefore $T$ is not a cutter operator whereas
\begin{displaymath}
0< \left< 0- \displaystyle\frac{T(x)+x}{2}, T(x)-x\right>= \left\{
\begin{array}{ll}
\frac{3}{8}x^2, & x\in Q \vspace{0.1cm}\\
\frac{4}{9}x^2, & x\in Q^{c}
\end{array} \right.
\end{displaymath}
for all $x\in \mathcal{H}\backslash Fix T.$ Thus $T$ is an sQNE
operator.
\end{example}

We next give the definition of {\it generalized relaxation} of an
operator, see, e.g., \cite{C2010,CC2011}.
\begin{definition}\label{def1}
Let $T: \mathcal{H} \rightarrow \mathcal{H}$ and
$\sigma:\mathcal{H}\rightarrow (0,\infty)$ be a {\it step size
function.} The generalized relaxation of $T$ is defined by
\begin{equation}\label{t4}
T_{\sigma,\lambda}(x)=x+\lambda \sigma(x)(T(x)-x),
\end{equation}
where  $\lambda$ is a relaxation parameter in $[0,2].$
\end{definition}
\begin{remark}\label{rem0}
If $\lambda\sigma(x) \geq 1$ for all $x\in \mathcal{H}$, then the
operator $T_{\sigma,\lambda}$ is called an {\it extrapolation} of
$T$. For $\sigma(x)=1$ we get the relaxed version of $T,$ namely,
$T_{1,\lambda}=:T_{\lambda}.$ Furthermore, it is clear that
$T_{\sigma,\lambda}(x)=x+\lambda(T_{\sigma}(x)-x)$ where
$T_{\sigma}=T_{\sigma,1}$ and $Fix~T_{\sigma,\lambda}=Fix~T$ for
any $\lambda\neq0.$
\end{remark}
It is shown in \cite{BC} that the relaxed version of a cutter
operator is cutter where the relaxation parameters lie in $[0,1].$
A similar result can be deduced for the family of sQNE operators as
follows.
\begin{lemma}\label{lemma2}
If $T:\mathcal{H}\rightarrow \mathcal{H}$ is an sQNE operator then
$T_{\alpha}=(1-\alpha)Id+\alpha T$ is an sQNE operator for any
$\alpha \in (0,1].$
\end{lemma}
\begin{proof}
Using Remark \ref{rem1}, we have
\begin{eqnarray}
\left<z-\frac{T_{\alpha}(x)+x}{2},T_{\alpha}(x)-x\right>
&=&\left<z-x-\alpha\frac{T(x)-x}{2},\alpha(T(x)-x)\right>\nonumber\\
&=&\alpha\left<z-x,T(x)-x\right>-\frac{\alpha^2}{2}\|T(x)-x\|^2\nonumber\\
&=&\alpha\left<z-\frac{T(x)+x}{2},T(x)-x\right>+\zeta_{\alpha}(x)\label{t5}
\end{eqnarray}
where
$\zeta_{\alpha}(x)=\frac{1}{2}\alpha(1-\alpha)\|T(x)-x\|^2\geq 0$
for all $x\in \mathcal{H}\backslash Fix T$ and $z\in Fix
T,~\alpha\in (0,1].$ Since $T$ is an sQNE operator, the inner
product in (\ref{t5}) is positive which leads to the desired
result in the lemma.
\end{proof}
\begin{definition}\label{def2}
An operator $T: \mathcal{H} \rightarrow \mathcal{H}$ is {\it
demi-closed} at $0$ if for any weakly converging sequence
$x^k\rightharpoonup y\in \mathcal{H}$ with $T(x^k)\rightarrow 0$
we have $T(y)=0$.
\end{definition}
\begin{remark}\label{rem2}
It is well known, see \cite[p. 108]{C2013} that for a nonexpansive
operator $T: \mathcal{H} \rightarrow \mathcal{H}$, the operator $T
- Id$ is demi-closed at $0$.
\end{remark}

\section{Main result}\label{main}
In this section we present our main result consisting of  an algorithm
and its convergence analysis. The algorithm is based on
generalized relaxation of an sQNE operator which is formed by
averaging of finitely many operators. These operators are
composition of finitely many sQNE operators. Therefore the
operators of averaging process, which are resulted by strings, can
be simultaneously computed.

The string averaging algorithmic scheme is first proposed in
\cite{string2001}. Their analysis was based on the projection
operators, whereas the algorithm is defined for any operators, for
solving consistent convex feasibility problems. Studying the
algorithm in a more general setting as Hilbert space is considered
by \cite{bauschke2004projection}.  The inconsistent case is analyzed by
\cite{inco-string} and they proposed a general algorithmic scheme
for string averaging method without any convergence analysis. A
subclass of the algorithm is studied under summable perturbation
in \cite{sumper2007,BRZ2008,DHC2009}. A dynamic version of the algorithm
is studied in \cite{CA2014,BRZ2018}. In \cite{Gordon2005} the string
averaging method is compared with other methods for sparse linear
systems.

Recently, a perturbation resilience iterative method with an
infinite pool of operators is studied in \cite{NA2015} which
answers all mentioned open problems of \cite[Case II of Sections
2.2 and 4]{string2001} for consistence case whereas these problems
are partially answered by \cite{dynsumper2013}. Also,
the proposed general algorithmic scheme of \cite[Algorithm 3.3]{inco-string},
which was presented without any convergence analysis, is extended
with a convergence proof by \cite{NA2015}. Furthermore, the
research work \cite{NA2015} extends the results of the
block-iterative method mentioned in \cite{NDH2012}, which solves
linear systems of equations, with a convergence analysis.

All the above mentioned research works are based on projection
operators. In \cite{ys2008}, the algorithm is studied for cutter
operators and the sparseness of the operators is used in averaging
process. In \cite{Crombez2002,Crombez2004}, the string averaging
method is used for finding common fixed point of strict
paracontraction operators. We next reintroduce the string
averaging algorithm as follows.
\begin{definition}\label{tarifstring}
The string $I^{t}=(i_1^t,i^t_2,...,i_{m_t}^t)$ is an ordered
subset of $I=\{1,2,...,m\}$ such that $\bigcup^{E}_{t=1}I^t=I$.
Define
\begin{eqnarray}\label{tarifstring2}
U_{t}&= &T_{i^{t}_{m_{t}}}...T_{i^{t}_{2}}T_{i^{t}_{1}},~t=1, 2,
\cdots, E\nonumber\\
T&=& \displaystyle\sum_{t=1}^{E} \omega_{t}U_{t}
\end{eqnarray}
where $\omega_{t}>0$ and $\sum_{t=1}^{E} \omega_{t}=1$. Here
$T_{i\in I}$ are operators from a Hilbert space $\mathcal{H}$ into
$\mathcal{H}.$
\end{definition}
In this paper, we assume that all $T_{i\in I}$ of Definition
\ref{tarifstring} are sQNE operators on $\mathcal{H}$ and the
averaging process (\ref{tarifstring2}) is a special case of
\cite{string2001}.

\begin{remark}\label{rem3}
Note that all $\{U_t\}_{t=1}^E$ and consequently the operator $T$
are sQNE, see \cite[Theorem 2.1.26]{C2013}.
\end{remark}
Before we introduce our main algorithm, which is based on the
Definition \ref{tarifstring}, we give an important special case
which is studied by many authors. Let $E=1$ in Definition
\ref{tarifstring} and $B_t\subseteq J=\{1, 2, \cdots, M\}$ such
that $\bigcup_{t=1}^m B_t=J$. Therefore, it results $U_1=T_m\cdots
T_1$ and $T=U_1$. Defining
\begin{equation}\label{tt}
T_t=(1-\lambda_t)Id+\lambda_t(\sum_{i\in B_t}\omega_i L_i),
\end{equation}
leads to a sequential block iterative method. This iterative
method with an infinite pool of cutter operators was studied in
\cite{com2001} and was generalized in \cite{BCk}. Here $\lambda_t$
is relaxation parameter and $L_i:\mathcal{H}\rightarrow
\mathcal{H}.$ It is well known, see Remark \ref{rem3}, that if all
$\{L_i\}_{i\in J}$ are sQNE operators then $\{T_t\}_{t=1}^m$ of
(\ref{tt}) are sQNE operators.

We now consider the following algorithm which is based on
generalized relaxation of (\ref{tarifstring2}).

\begin{algorithm}[H]
\begin{algorithmic}[1]
\STATE {\bf Initialization:} Choose an arbitrary initial guess $x^{0}\in \mathcal{H}$.
\FORALL {$k\geq 0$}
\STATE{$x^{k+1}=T_{\sigma,\lambda_k}(x^k)$}.
\ENDFOR
\end{algorithmic}
\caption{}\label{algt1}
\end{algorithm}
We next show that the generalized relaxation operator $T_{\sigma,
1}=T_{\sigma}$ is an sQNE operator under a condition on
$\sigma(x)$. Indeed
\begin{lemma}\label{lemma3}
Let $T_{\sigma}$ be a generalized relaxation of $T=\sum_{t=1}^{E}
\omega_{t}U_{t}.$ Then $T_{\sigma}$ is an sQNE operator if $0
<\sigma(x)<\frac{\sum^{E}_{t=1}w_{t}\|U_{t}(x)-x\|^{2}}{\|T(x)-x\|^{2}}$
where $x\in \mathcal{H}\backslash Fix T.$ Furthermore, the step
size function
\begin{displaymath}
\sigma_{max}(x) := \left\{ \begin{array}{ll}
\frac{\sum^{E}_{t=1}w_{t}\|U_{t}(x)-x\|^{2}}{\|T(x)-x\|^{2}}, & x\in \mathcal{H}\backslash Fix T \vspace{0.1cm}\\
1, & x\in Fix T
\end{array} \right.
\end{displaymath}
is bounded below by $1.$
\end{lemma}
\begin{proof}
For $z\in Fix~T_{\sigma}$ and $x\in \mathcal{H}\backslash
Fix~T_{\sigma}$ we obtain
\begin{eqnarray}
\left<z-\frac{T_{\sigma}(x)+x}{2},T_{\sigma}(x)-x\right>&=&
\left<z-x-\frac{\sigma(x)(T(x)-x)}{2},T_{\sigma}(x)-x\right>\nonumber\\
&=&\left<z-x,T_{\sigma}(x)-x\right>+\nonumber\\
&&~~~~~~~-\left<\frac{\sigma(x)(T(x)-x)}{2},T_{\sigma}(x)-x\right>\nonumber\\
&=&\sigma(x)\left<z-x,T(x)-x\right>+\nonumber\\
&&~~~~~~~-\frac{1}{2}\sigma^2(x)\|T(x)-x\|^{2}.\label{t14}
\end{eqnarray}
Since $\sum^{E}_{t=1}w_{t}=1,$
$Fix~T_{\sigma}=Fix~T=\bigcap_{t=1}^E Fix~U_t$  and all
$\{U_{t}\}_{t=1}^E$ are sQNE operators, see Remark \ref{rem3}, we
have
\begin{eqnarray}
\left<z-x,T(x)-x\right> &=&\left<z-x,\sum^{m}_{t=1}w_{t}U_{t}(x)-x\right>\nonumber\\
&=& \sum^{E}_{t=1}w_{t} \left<z-x,U_{t}(x)-x\right>\nonumber\\
&=& \sum^{E}_{t=1}w_{t} \left<z-\frac{U_{t}(x)+x}{2}+\frac{U_{t}(x)-x}{2},U_{t}(x)-x\right>\nonumber\\
&>&\frac{1}{2}\sum^{E}_{t=1}w_{t}\|U_{t}(x)-x\|^{2} \label{t8}\\
&\geq& \frac{1}{2}\|\sum^{E}_{t=1}w_{t}(U_{t}(x)-x)\|^{2} (\textrm{by convexity of $\|.\|^2$})\nonumber\\
&=&\frac{1}{2}\|\sum^{E}_{t=1}w_{t}U_{t}(x)-x\|^{2}=\frac{1}{2}\|T(x)-x\|^{2}.\label{t11}
\end{eqnarray}
Combining (\ref{t14}) and (\ref{t8}), we obtain
$\left<z-\frac{T_{\sigma}(x)+x}{2},T_{\sigma}(x)-x\right>>0$ which
shows $T_{\sigma}$ is an sQNE operator. Furthermore, the lower
bound for $\sigma_{max}$ is derived by (\ref{t8}) and (\ref{t11}).
\end{proof}

\begin{theorem}\label{theorem2}
Let $\sigma=\sigma_{max}$ be step-size function and
$\lambda_{k}\in[\varepsilon,1-\varepsilon]$ for an arbitrary
constant $\varepsilon\in(0,\frac{1}{2})$. The generated sequence
of Algorithm \ref{algt1} weakly converges to a point in $Fix T$,
if one of the following conditions is satisfied:
\begin{itemize}
\item[(i)] $T-Id$ is demi-closed at $0$, or
\item[(ii)] $U_{t} -Id$ are demi-closed at $0$, for all
$t=1, 2, \cdots, E.$
\end{itemize}
\end{theorem}
\begin{proof}
We first show $\{x^k\}$ is Fej\'{e}r monotone with respect to $Fix
T,$ namely, $\|x^{k+1}-z\|\leq \|x^{k}-z\|$ for any $z\in Fix T.$
Using Remark \ref{rem0} and for $z\in Fix T, x^k\in
\mathcal{H}\backslash Fix T$ we have
\begin{eqnarray}\label{aslih2}
\|x^{k+1}-z\|^2 &=& \|T_{\sigma,\lambda_k}(x^k)-z\|^2\nonumber\\
&=&\|x^{k}+\lambda_{k}\sigma(x^k)(T(x^k)-x^{k})-z\|^2\nonumber\\
&=& \|x^{k}+\lambda_{k}(T_{\sigma}(x^k)-x^{k})-z\|^2\nonumber\\
&=&\|x^{k}-z\|^2 +\xi_k
+\lambda^{2}_{k}\|T_{\sigma}(x^k)-x^{k}\|^2
\end{eqnarray}
where $\xi_k=2 \lambda_{k}\left<x^k
-z,T_{\sigma}(x^k)-x^{k}\right>.$ Using  Lemma \ref{lemma3}, we
obtain
$$\left<\frac{T_{\sigma}(x^k)+x^k}{2}
-z,T_{\sigma}(x^k)-x^{k}\right>\leq 0$$ and therefore
\begin{eqnarray}\label{aslih22}
\xi_k&=&2\lambda_{k}\left<x^k-
\frac{T_{\sigma}(x^k)}{2} + \frac{T_{\sigma}(x^k)}{2} -z,T_{\sigma}(x^k)-x^{k}\right>\nonumber\\
&\leq& -\lambda_{k}\|T_{\sigma}(x^k)-x^{k}\|^2.
\end{eqnarray}
Now using (\ref{aslih2}), (\ref{aslih22}) and Lemma \ref{lemma3},
we obtain
\begin{eqnarray}
\|x^{k+1}-z\|^2 &\leq&\|x^{k}-z\|^2 -\lambda_{k}\|T_{\sigma}(x^k)-x^{k}\|^2+\lambda^{2}_{k}\|T_{\sigma}(x^k)-x^{k}\|^2\nonumber\\
&=& \|x^{k}-z\|^2  - (\lambda_{k}-\lambda^{2}_{k})\sigma^{2}(x^k)\|T(x^k)-x^{k}\|^2 \label{t10}\\
&\leq& \|x_{k}-z\|^2 -\lambda_{k}(1-\lambda_{k})\|T(x^k)-x^{k}\|^2 \nonumber\\
&=&\|x_{k}-z\|^2 -\lambda_{k}(1-\lambda_{k})
\left\|\sum^{E}_{t=1}w_{t}U_{t}(x^k)-x^{k}\right\|^2\label{t7}.
\end{eqnarray}
Therefore, the sequence $\left\{\|x^{k}-z\|\right\}$ decreases and
consequently $\left\{x^{k}\right\}$ is bounded. Using (\ref{t7}),
one easily gets
\begin{equation}\label{demiclosedT}
\|T(x^k)-x^k\|\rightarrow 0
\end{equation}
as $k\rightarrow \infty.$ Using \cite[Theorem 2.16 (ii)]{BB}, the
sequence $\left\{x^{k}\right\}$ has at most one weak cluster point
$x^*\in \mathcal{H}.$ Assume $\left\{x^{n_k}\right\}$ be a
subsequence of $\left\{x^{k}\right\}$ which weakly converges to
$x^*.$ Using (\ref{demiclosedT}) and demi-closedness of $T-Id$ at
$0,$ we have $x^{*}\in FixT$ which implies all weak cluster points
of $\left\{x^{k}\right\}$ lie in $FixT.$ Again using \cite[Theorem
2.16 (ii)]{BB} we conclude that the sequence
$\left\{x^{k}\right\}$ weakly converges to $x^{*}.$

We next assess the second part of theorem. Using (\ref{t8}) and
(\ref{t11}), one may rewrite (\ref{t10}) as
\begin{eqnarray}
\|x^{k+1}-z\|^2 &<&\|x^{k}-z\|^2  -
\lambda_{k}(1-\lambda_{k})\frac{\left(\sum^{E}_{t=1}w_{t}\|U_{t}(x^k)-x^k\|^{2}\right)^2}{\|T(x^k)-x^{k}\|^2}\nonumber\\
&\leq &\|x^{k}-z\|^2  -
\lambda_{k}(1-\lambda_{k})\sum^{E}_{t=1}w_{t}\|U_{t}(x^k)-x^k\|^{2}\label{t12}
\end{eqnarray}
which gives $\|U_{t}(x^k)-x^k\|\rightarrow 0$ as $k\rightarrow
\infty$ for $t=1, \cdots, E.$ As in the first part of the proof,
let $\left\{x^{n_k}\right\}$ be a subsequence of
$\left\{x^{k}\right\}$ which weakly converges to a weak cluster
point $x^*\in \mathcal{H}.$ Similarly, demi-closedness of
$\{U_t-Id\}_{t=1}^E$ leads to $x^*\in Fix U_t$ for all $t=1,
\cdots, E.$ Thus $x^*\in\bigcap_{t=1}^E Fix U_t$ and we obtain
that, using \cite[Theorem 2.16 (ii)]{BB}, the sequence
$\left\{x^{k}\right\}$ weakly converges to $x^{*}.$
\end{proof}

We next compare results of \cite[Theorem 9]{CC} with the above
theorem. Therefore we consider one string, i.e., $E=1$, see
Definition \ref{tarifstring}. Note that in the case $E=1$ we have
$\sigma_{max}=1$ which means no generalized relaxation process is
affected, but we still have a better error reduction than what was
reported in \cite[Theorem 9]{CC}. Putting $e_k:=\|x^k-z\|^2$ in
(\ref{t7}) leads to the difference of successive errors as follows
\begin{eqnarray}
e_{k}-e_{k+1}>\lambda_{k}(1-\lambda_{k})
\|T(x^k)-x^{k}\|^2.\label{t15}
\end{eqnarray}

Based on \cite[Theorem 9]{CC}, the inequality (\ref{t15}) is
replaced by
\begin{eqnarray}
e_{k}-e_{k+1}>\frac{\lambda_{k}(2-\lambda_{k})}{4m^2}
\|T(x^k)-x^{k}\|^2\label{t150}
\end{eqnarray}
where $\lambda_k\in[\varepsilon,2-\varepsilon]$ for an arbitrary
constant $\varepsilon\in(0,1).$ Since the set of cutter operators
is a subset of sQNE operators and comparing lower bounds
(\ref{t15}) and (\ref{t150}), we conclude that the generalized
relaxation of a cutter operator has the ability of faster
reduction in error.

\section{Applications}\label{apl}
In this section we reintroduce some state-of-the-art iterative
methods which are based on sQNE operators. The section consists of
 two parts. The First one covers a wide class of block iterative
projection methods
 for solving linear equations and/or inequalities.
 In the next part, the operators
$U_{t}$ of Definition \ref{tarifstring} are replaced by the
 operators which are based on the parallel subgradient
 projection method. 
\subsection{Block iterative method}
First, we begin with block iterative methods which are used for
solving linear systems of equations (inequalities). Let
$A\in\mathbb{R}^{m\times n}$ and $b\in\mathbb{R}^m$ be given. We
assume the consistent linear system of equations
\begin{equation}\label{block1ap}
Ax = b.
\end{equation}
Let $A$ and $b$ be partitioned into $p$ row-blocks
$\{A_t\}_{t=1}^p$ and $\{b^t\}_{t=1}^p$ respectively. When $p = 1$
the method is called fully simultaneous iteration. On the other
hand, for $p = m$ we get a fully sequential iteration. Consider
the following algorithm which are studied in several research
works as
\cite{NDH2012,CGG2001,CE2002,CEHN2008,DHC2009,EHL1981,E1980,EN2009,JW2003}.
\begin{algorithm}[H]
\begin{algorithmic}[1]
\STATE {\bf Initialization:} Choose an arbitrary initial guess $x^{0}\in \mathbb{R}^{n}$
\FORALL {$k\geq 0$}
\STATE{$x^{k,0}=x^k$}
\FOR{$t=1$ \TO p}
\STATE{\begin{eqnarray}
x^{k,t}&=&x^{k,t-1}+\lambda_t A_{t}^T M_{t}(b^t-A_t
x^{k,t-1})\nonumber\\
&=&T_t(x^{k,t-1})\nonumber
\end{eqnarray}}
\ENDFOR
\STATE{$x^{k+1}=x^{k,p}$}
\ENDFOR
\end{algorithmic}
\caption{Sequential Block Iteration}\label{alg:alg2ap}
\end{algorithm}

where
\begin{equation}\label{t8ap}
T_t(x)=x+\lambda_t A_{t}^T M_{t}(b^t-A_t x).
\end{equation}
Here $\{\lambda_t\}_{t=1}^p$ and $\{M_t\}_{t=1}^p$ are relaxation
parameters and symmetric positive definite weight matrices
respectively. If $0<\varepsilon\leq
\lambda_t\leq\frac{2-\varepsilon}{\rho(A_t^T M_t A_t)} \textrm{
for } t=1,\dots, p $, where $\rho(B)$ denotes the spectral radius
of $B,$ then using \cite[Lemmas 3 and 4]{NDH2012} we conclude that
the operator $T_t$ in (\ref{t8ap}) is not only an sQNE operator but
also is nonexpansive. Since the set of nonexpansive operators is
closed with respect to composition and convex combination of
operators, we get that the operators $T-Id$ and
$\{U_t-Id\}_{t=1}^E$ are nonexpansive. Therefore, based on Remark
\ref{rem2}, both conditions $(i)$ and $(ii)$ of Theorem
\ref{theorem2} are satisfied. Therefore $T_t$ can be applied in
Algorithm \ref{algt1} without losing convergence result of Theorem
\ref{theorem2}.

\subsection{Subgradient projection}\label{subgradient:section}
We next replace $T_t$ of Definition \ref{tarifstring} by the
subgradient projection operator and verify the conditions of
Theorem \ref{theorem2}. Let $i\in J=\{1,2,\cdots, M\}$, the index
set, and $g_i :D \subseteq \mathbb{R}^{n} \rightarrow \mathbb{R}$
be convex functions. We consider consistent system of convex
inequalities
\begin{equation}\label{funcap}
g_i(x)\leq 0,~~~\textrm{ for } i\in J
\end{equation}
Let $g^{+}_{i}(x)= max \{ 0 , g_{i}(x)\}$, and denote the solution
set of (\ref{funcap}) by $S=\{x | g_{i}(x)\leq 0,~i\in J \}.$ Then
$g^{+}_{i}(x)$ is a convex function and
\begin{equation}\label{moadelap}
S=\{x | g^{+}_{i}(x)= 0,~i\in J\}.
\end{equation}

We divide the problem (\ref{moadelap}) into $m$ blocks as follows.
Set $B_{t} \subseteq J$ such that $\bigcup^{m}_{t=1} B_{t}=J$. Let
$\ell_{i}(x)$ and $\partial g_{i}^{+}(x)$ denote subgradient and
set of all subgradients of $g_i$ at $x$ respectively. Here a
vector $t\in\mathbb{R}^n$ is called subgradient of a convex
function $g$ at a point $y\in \mathbb{R}^n$ if $\langle t,
x-y\rangle \leq g(x)-g(y)$ for every $x\in \mathbb{R}^n.$ It is
known that the subgradient of a convex function always exist. We
first consider the following operators which are used in cyclic
subgradient projection method, see \cite{cl81},
\begin{equation}\label{tt1}
T_{t}(x)=x-\mu_{t}\frac{g_{t}^{+}(x)}{\|\ell_{t}(x)
\|^2}\ell_{t}(x).
\end{equation}
Based on analysis of \cite[Section 4]{CC}, see also \cite[Theorem
4.2.7]{C2013}, $T_t-Id$ is demi-closed at $0$ under a mild
condition which holds for any finite dimensional spaces. Again
using \cite[Theorem 4.2.7]{C2013}, one easily obtains that
$T_t-Id$ is demi-closed at $0$ where
\begin{equation}\label{general1ap}
T_{t}(x)=x-\mu_{t}\sum_{i\in B_{t}} \omega_{i}
\frac{g_{i}^{+}(x)}{\|\ell_{i}(x) \|^2}\ell_{i}(x)
\end{equation}
with $\sum\limits_{i \in {B_t}} \omega _i = 1,~\omega _i> 0$ and
defined $\mu_{t}$ in (\ref{mu}). We next show both operators of
(\ref{tt1}) and (\ref{general1ap}) are sQNE operators. One easily
gets
\begin{equation}\label{t1ap}
\|T_t(x)-z\|^2\leq \|x-z\|^2-2\mu_t \sum_{i\in B_t}\omega_i
\frac{(g_{i}^{+}(x))^2}{\|\ell_{i}(x) \|^2}+\mu_t^2 \|\sum_{i\in
B_t}\omega_i \frac{g_{i}^{+}(x)}{\|\ell_{i}(x)
\|^2}\ell_{i}(x)\|^2.
\end{equation}
By choosing
\begin{equation}\label{mu}
\mu_t=\frac{\sum_{i\in B_t}\omega_i
\frac{(g_{i}^{+}(x))^2}{\|\ell_{i}(x) \|^2}}{\|\sum_{i\in
B_t}\omega_i \frac{g_{i}^{+}(x)}{\|\ell_{i}(x)
\|^2}\ell_{i}(x)\|^2},
\end{equation}
which minimizes the right hand side of (\ref{t1ap}), and some
calculations we obtain
\begin{equation}\label{t2ap}
\|T_t(x)-z\|^2\leq \|x-z\|^2-\frac{\left(\sum_{i\in B_t}\omega_i
g_{i}^{+}(x)\right)^2}{N}
\end{equation}
where $N>0$, $x\in \mathbb{R}^n\backslash Fix T_t$ and $z\in Fix
T_t.$ Using (\ref{t2ap}) and this fact that $\{x | g^{+}_{i}(x)=
0,~i\in B_t\}\subseteq Fix T_t,$ we get $\|T_t(x)-z\|< \|x-z\|$
which means all $T_t$ of (\ref{tt1}) and (\ref{general1ap}) are
sQNE operators. Thus, for both cases (\ref{tt1}) and
(\ref{general1ap}) where $I^t=(t)$, see Definition
\ref{tarifstring}, the conditions ($i$) and ($ii$) of Theorem
\ref{theorem2} are satisfied.
\section{Numerical Results}\label{num}
We will report from tests using classical examples, the field of
image reconstruction from projections and randomly produced
examples. Throughout this section, we assume the relaxation
parameter of Algorithm \ref{algt1} is equal to one ($\lambda_k=1$)
and all strings are of length one.

We first report numerical tests for classical examples taken from
\cite{dos87} with larger sizes. In the first part of our numerical
tests, we assume four strings (i.e., $E=4$, see Definition
\ref{tarifstring}), $m=4$ blocks which are contained by $50$
convex functions, i.e., number of elements in each $B_t$ is $50$,
and therefore $M=200,$ see Section \ref{subgradient:section} for
the notations. Also in each block we use parallel subgradient
projection operator, defined by (\ref{general1ap}), with equal
weights, i.e., $w_i=1/50.$ The relaxation parameter $\mu_t$ is
defined by (\ref{mu}). In Table \ref{tab:0}, we give the results
related to examples (top-down) {\it Extended Powell singular
function,} {\it Chained Wood function,} {\it Extended Rosenbrock
function,} {\it Broyden tridiagonal function,} {\it Penalty
function} and {\it Variably dimensioned function} from
\cite{dos87} and \cite{testp}, see Section \ref{appendix} for more
details. 
In the table, where the number of variables are denoted by $n,$ we
report the number of iterations using extrapolation (ue) and
without extrapolation (we), i.e. $\sigma_{max}=1.$ In Table
\ref{tab:0}, the first and the second component of a pair shows
the number of iterations using stopping criteria $g_i^+(x)\leq
10^{-1}$ and $g_i^+(x)\leq 10^{-4}$ for
 all $i\in J$ respectively. In order to avoid ``division by zero'' when calculating
$\sigma_{max}$, we involve the criterion $\|T(x^k)-x^{k}\|^2\leq
10^{-10}$ where the operator $T$ is defined in Definition
\ref{tarifstring}. Results of Table \ref{tab:0} show that the
extrapolation strategy gives better results except the second
example, i.e, Chained Wood function, see the second row of Table
\ref{tab:0}. As it is mentioned in \cite{CC}, using extrapolation
strategy has local acceleration property and it does not guarantee
overall acceleration of the process. As it is seen in Table
\ref{tab:0}, changing the stopping criteria gives proper result
for this example. The reason may be due to having ``tiny'' set of
solution or local acceleration property of the extrapolation.

\begin{table}
\caption{ Results of classical examples \label{tab:0}}
\begin{tabular}{ccc}
\hline\noalign{\smallskip}
$n$ & (ue) & (we) \\
\noalign{\smallskip}\hline\noalign{\smallskip}
102 & (16,26) &(28,1054) \\
~68 &(4,189) &(29,127) \\
101 & (5,6) &(24,35) \\
200 &(3,3) & (23,38)\\
199 &(7,10) &(39,63) \\
198 & (10,16) & (40,53)\\
 \noalign{\smallskip} \hline
\end{tabular}
\end{table}
The second test is taken from the field of image reconstruction
from projections using the SNARK93 software package \cite{Browne}.
We work with the standard head phantom from \cite{herman1980}. The
phantom is discretized into $63\times 63$ pixels which satisfies
the linear system of equations $Ax = b$. We are using $16$
projections with $99$ rays per projection.

The resulting projection matrix $A$ has dimension $1376 \times
3969$, so that the system of equations is highly underdetermined.

Let $A$ and $b$ be partitioned into $16$ row blocks
$\{A_t\}_{t=1}^{16}$ and $\{b_t\}_{t=1}^{16}$ respectively. We use
Cimmino's $M-$matrix in Algorithm \ref{alg:alg2ap}. The relaxation
parameters, in Algorithm \ref{alg:alg2ap}, are chosen such that
each of them minimizes  $M_t-$weighted norm of residual in each
block, the convergence analysis of this strategy is investigated
in \cite{cauchy}. Also we consider sixteen strings, i.e., $E=16$
and use $U_t=T_t$ for $t=1,\cdots,16,$ see (\ref{t8ap}) and
Definition \ref{tarifstring}. Fig. \ref{phantom2} demonstrates
iteration history for the relative error using extrapolation and
without extrapolation within $50$ iterations.
 \begin{figure}
  \centerline{\includegraphics[width=9.cm,height=6cm]{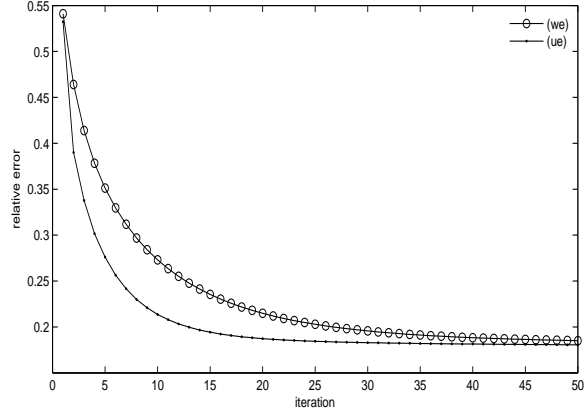}}
    \caption{ The history of relative error within $50$ iterations.
}\label{phantom2}
\end{figure}

We next examine $100$ nonlinear system of inequalities with $300$
variables which are produced randomly and each of them is
contained by $200$ convex functions. All randomly produced
matrices and vectors have entries in $[-10, 10].$ First we explain
how one of them is made. After generating the matrices
$G_i\in\mathbb{R}^{300\times 300}$ and the vectors
$c_i\in\mathbb{R}^{300}$ for $i=1,\cdots, 200$, we define the
following convex functions
$$f_i(x)=x^T G_i^T G_i x+c_i^Tx+d_i.$$
The vectors $d_i$ are calculated such that $f_i(y)\leq 0$ where
$y=(1, \cdots, 1)^T.$ Therefore the solution set
$$S=\left\{x|f_i(x)\leq 0,~~i=1,\cdots, 200\right\}$$
has at least one point. Similar to the first part of our tests, we
assume four strings, i.e. $E=4$, which are contained by $50$
convex functions, i.e., number of elements in each $B_t$ is $50$,
see Section \ref{subgradient:section}. Also in each block we use
parallel subgradient projection operator with equal weights and
the relaxation parameter $\mu_t$ is defined by (\ref{mu}).

 In addition, the iteration
is stopped when $f_i^+(x)\leq 10^{-4}$ for all $i\in J$ or
$\|T(x^k)-x^{k}\|^2\leq 10^{-10}.$ Table \ref{tab:2} explains the
mean value of iteration numbers. As it is seen, the iteration
number is reduced by (ue).

\begin{table}
\caption{ Results of $100$ nonlinear system of
inequalities \label{tab:2}}
\begin{tabular}{ccc}
\hline\noalign{\smallskip} method  & iteration averaging \\
\noalign{\smallskip}\hline\noalign{\smallskip}
(ue) & 8.49\\
(we)& 30.63\\\noalign{\smallskip}\hline
\end{tabular}
\end{table}

\section{Conclusion}
In this paper we consider an extrapolated version of string averaging
method, which is based on strictly quasi-nonexpansive operators.
Our analysis indicates that the generalized relaxation of cutter
operators is inherently able to provide more acceleration comparing
with results of \cite{CC}. As a special case of our algorithm, i.e.,
Algorithm \ref{algt1}, we consider a wide class of iterative
methods for solving linear systems of equations (inequalities) and
the subgradient projection method for solving nonlinear convex
feasibility problems. Our numerical tests show that using
extrapolation strategy we are able to reduce the number of iterations
to achieve a feasible point.



\vspace*{-10pt}
\section*{Appendix}\label{appendix}

In this part the classical test problems, used in Section
\ref{num}, will be explained. For positive integers $i$ and $l$,
we use the notations $div(i, l)$ for integer division, i.e.,
 the largest integer not greater than $i/l$, and $mod(i, l)$ for the
remainder after integer division, i.e., $mod(i, l) = l(i/l -
div(i, l)).$ Here $n$ denotes the number of variables.

\begin{enumerate}\footnotesize
  \item Extended Powell singular function
\begin{align*}
&g_{i}(x)=x_{j}+10x_{j+1},& mod(i,4)=1\\
&g_{i}(x)=\sqrt{5}(x_{j+2}-x_{j+3}),& mod(i,4)=2\\
&g_{i}(x)=(x_{j+1}-2x_{j+2})^{2},& mod(i,4)=3\\
&g_{i}(x)=\sqrt{10}(x_{j}-x_{j+3})^{2},& mod(i,4)=0\\
\hline
&x^{0}_{l}=3,&mod(l,4)=1\\
&x^{0}_{l}=-1,&mod(l,4)=2\\
&x^{0}_{l}=0,&mod(l,4)=3\\
&x^{0}_{l}=1,&mod(l,4)=0\\
\hline
&n=102&j=2div(i+3,4)-1.\\
\hline
\end{align*}

  \item {Chained Wood function}
  \begin{align*}
&g_{i}(x)=10(x_{j-1}^{2}-x_{j}),& mod(i,6)=1\\
&g_{i}(x)=x_{j-1}-1,& mod(i,6)=2\\
&g_{i}(x)=\sqrt{90}(x_{j+1}^{2}-x_{j+2}),& mod(i,6)=3\\
&g_{i}(x)=x_{j+1}-1,& mod(i,6)=4\\
&g_{i}(x)=\sqrt{10}(2-x_{j}-x_{j+2}),& mod(i,6)=5\\
&g_{i}(x)=\frac{1}{\sqrt{10}}(x_{j+2}-x_{j}),& mod(i,6)=0\\
\noalign{\smallskip}\hline
&x^{0}_{l}=-3,&mod(l,2)=1, l\leq 4\\
&x^{0}_{l}=-1,&mod(l,2)=0, l\leq 4\\
&x^{0}_{l}=-2,&mod(l,2)=1, l>4\\
&x^{0}_{l}=0,&mod(l,2)=0, l>4\\
\hline
&n=68&j=2(div(i,6)+1).\\
\hline
\end{align*}
  \item {Extended Rosenbrock function}
  \begin{align*}
&g_{i}(x)=10(x_{j}^2-x_{j+1}),& mod(i,2)=1\\
&g_{i}(x)=x_{j}-1,& mod(i,2)=0\\
\hline
&x^{0}_{l}=-1.2,&mod(l,2)=1\\
&x^{0}_{l}=-1,&mod(l,2)=0\\
\hline
&n=101&j=div(i+1,2).\\
\hline
\end{align*}
  \item {Broyden tridiagonal function}
    \begin{align*}
&g_{i}(x)=(3-2x_{j})-x_{j-1}-2x_{j+1}+1&\\
\hline
&x^{0}_{l}=-1,&l\geq 1\\
\hline
&n=200&x_{0}=x_{n+1}=0.\\
\hline
\end{align*}
  \item {Penalty function 1}
  \begin{align*}
&g_{i}(x)=x_{j}-1,&1\leq k \leq 199\\
&g_{i}(x)=\frac{1}{\sqrt{1000}}\sum_{i=1}^{n}(x_{j}^{2}-\frac{1}{4}),&k=200\\
\hline
&x^{0}_{l}=l,&l\geq 1\\
\hline
&n=199&x_{0}=x_{n+1}=0.\\
\hline
\end{align*}
  \item {Variably dimensioned function}
  \begin{align*}
&g_{i}(x)=x_{i}-1,& 1\leq i\leq n\\
&g_{i}(x)=\sum_{j=1}^{n} j(x_{j}-1),&i=n+1\\
&g_{i}(x)=\left(\sum_{j=1}^{n} j(x_{j}-1)^2\right)^2,& i=n+2\\
\noalign{\smallskip}\hline
&x^{0}_{l}=1-\frac{l}{198},&l\geq 1\\
\noalign{\smallskip}\hline
&n=198.\\
\hline
\end{align*}
\end{enumerate}


\end{document}